\documentclass[a4paper,10pt]{amsart}

\usepackage{amsmath,amssymb,amsfonts,latexsym}

\newtheorem{theorem}{Theorem}[section]
\newtheorem{lemma}[theorem]{Lemma}
\newtheorem{corollary}[theorem]{Corollary}

\newtheorem{definition}[theorem]{Definition}

\newcommand{\mymod}[3]{#1 \equiv #2 \kern -0.5em \pmod{#3}}
\newcommand{\mynotmod}[3]{#1 \not \equiv #2 \kern -0.6em \pmod{#3}}

\begin{document}

\title[On the third-order Jacobsthal matrix sequences]{On the third-order Jacobsthal and third-order Jacobsthal-Lucas sequences and their matrix representations}

\author[G. Cerda-Morales]{Gamaliel Cerda-Morales}
\address{Instituto de Matem\'aticas, Pontificia Universidad Cat\'olica de Valpara\'iso, Blanco Viel 596, Valpara\'iso, Chile.}
\email{gamaliel.cerda.m@mail.pucv.cl}


\begin{abstract}
In this paper, we first give new generalizations for third-order Jacobsthal $\{J_{n}^{(3)}\}_{n\in \mathbb{N}}$ and third-order Jacobsthal-Lucas $\{j_{n}^{(3)}\}_{n\in \mathbb{N}}$ sequences for Jacobsthal and Jacobsthal-Lucas numbers. Considering these sequences, we define the matrix sequences which have elements of $\{J_{n}^{(3)}\}_{n\in \mathbb{N}}$ and $\{j_{n}^{(3)}\}_{n\in \mathbb{N}}$. Then we investigate their properties.

\vspace{2mm}


\vspace{2mm}


\end{abstract}



\maketitle


\section {Introduction}


The Jacobsthal numbers have many interesting properties and applications in many fields of science (see, e.g., \cite{Ba}). The Jacobsthal numbers $J_{n}$ are defined by the recurrence relation
\begin{equation}\label{e1}
J_{0}=0,\ J_{1}=1,\ J_{n+1}=J_{n}+2J_{n-1},\ n\geq1.
\end{equation}
Another important sequence is the Jacobsthal-Lucas sequence. This sequence is defined by the recurrence relation $j_{0}=2,\ j_{1}=1,\ j_{n+1}=j_{n}+2j_{n-1},\ n\geq1$. (see, \cite{Hor3}).

In \cite{Cook-Bac} the Jacobsthal recurrence relation (\ref{e1}) is extended to higher order recurrence relations and the basic list of identities provided by A. F. Horadam \cite{Hor3} is expanded and extended to several identities for some of the higher order cases. In particular, third order Jacobsthal numbers, $\{J_{n}^{(3)}\}_{n\geq0}$, and third order Jacobsthal-Lucas numbers, $\{j_{n}^{(3)}\}_{n\geq0}$, are defined by
\begin{equation}\label{ec:5}
J_{n+3}^{(3)}=J_{n+2}^{(3)}+J_{n+1}^{(3)}+2J_{n}^{(3)},\ J_{0}^{(3)}=0,\ J_{1}^{(3)}=J_{2}^{(3)}=1,\ n\geq0,
\end{equation}
and 
\begin{equation}\label{ec:6}
j_{n+3}^{(3)}=j_{n+2}^{(3)}+j_{n+1}^{(3)}+2j_{n}^{(3)},\ j_{0}^{(3)}=2,\ j_{1}^{(3)}=1,\ j_{2}^{(3)}=5,\ n\geq0,
\end{equation}
respectively.

The following properties given for third order Jacobsthal numbers and third order Jacobsthal-Lucas numbers play important roles in this paper (see \cite{Cer,Cer1,Cook-Bac}). 
\begin{equation}\label{e4}
3J_{n}^{(3)}+j_{n}^{(3)}=2^{n+1},
\end{equation}
\begin{equation}\label{e5}
j_{n}^{(3)}-3J_{n}^{(3)}=2j_{n-3}^{(3)},
\end{equation}
\begin{equation}\label{ec5}
J_{n+2}^{(3)}-4J_{n}^{(3)}=\left\{ 
\begin{array}{ccc}
-2 & \textrm{if} & \mymod{n}{1}{3} \\ 
1 & \textrm{if} & \mynotmod{n}{1}{3}%
\end{array}%
\right. ,
\end{equation}
\begin{equation}\label{e6}
j_{n}^{(3)}-4J_{n}^{(3)}=\left\{ 
\begin{array}{ccc}
2 & \textrm{if} & \mymod{n}{0}{3} \\ 
-3 & \textrm{if} & \mymod{n}{1}{3}\\ 
1 & \textrm{if} & \mymod{n}{2}{3}%
\end{array}%
\right. ,
\end{equation}
\begin{equation}\label{e7}
j_{n+1}^{(3)}+j_{n}^{(3)}=3J_{n+2}^{(3)},
\end{equation}
\begin{equation}\label{e8}
j_{n}^{(3)}-J_{n+2}^{(3)}=\left\{ 
\begin{array}{ccc}
1 & \textrm{if} & \mymod{n}{0}{3} \\ 
-1 & \textrm{if} & \mymod{n}{1}{3} \\ 
0 & \textrm{if} & \mymod{n}{2}{3}%
\end{array}%
\right. ,
\end{equation}
\begin{equation}\label{e9}
\left( j_{n-3}^{(3)}\right) ^{2}+3J_{n}^{(3)}j_{n}^{(3)}=4^{n},
\end{equation}
\begin{equation}\label{e10}
\sum\limits_{k=0}^{n}J_{k}^{(3)}=\left\{ 
\begin{array}{ccc}
J_{n+1}^{(3)} & \textrm{if} & \mynotmod{n}{0}{3} \\ 
J_{n+1}^{(3)}-1 & \textrm{if} & \mymod{n}{0}{3}%
\end{array}%
\right. 
\end{equation}
and
\begin{equation}\label{e12}
\left( j_{n}^{(3)}\right) ^{2}-9\left( J_{n}^{(3)}\right)^{2}=2^{n+2}j_{n-3}^{(3)}.
\end{equation}

Using standard techniques for solving recurrence relations, the auxiliary equation, and its roots are given by 
$$x^{3}-x^{2}-x-2=0;\ x = 2,\ \textrm{and}\ x=\frac{-1\pm i\sqrt{3}}{2}.$$ 

Note that the latter two are the complex conjugate cube roots of unity. Call them $\omega_{1}$ and $\omega_{2}$, respectively. Thus the Binet formulas can be written as
\begin{equation}\label{b1}
J_{n}^{(3)}=\frac{2}{7}2^{n}-\frac{3+2i\sqrt{3}}{21}\omega_{1}^{n}-\frac{3-2i\sqrt{3}}{21}\omega_{2}^{n}=\frac{1}{7}\left(2^{n+1}-V_{n}^{(3)}\right)
\end{equation}
and
\begin{equation}\label{b2}
j_{n}^{(3)}=\frac{8}{7}2^{n}+\frac{3+2i\sqrt{3}}{7}\omega_{1}^{n}+\frac{3-2i\sqrt{3}}{7}\omega_{2}^{n}=\frac{1}{7}\left(2^{n+3}+3V_{n}^{(3)}\right),
\end{equation}
respectively. Here, the sequence $\{V_{n}^{(3)}\}_{n\geq 0}$ is defined by
$$V_{n}^{(3)}=\left\{ 
\begin{array}{ccc}
2 & \textrm{if} & \mymod{n}{0}{3} \\ 
-3 & \textrm{if} & \mymod{n}{1}{3} \\ 
1 & \textrm{if} & \mymod{n}{2}{3}%
\end{array}%
\right. .$$

In \cite{Ci1,Ci2}, the authors defined a new matrix generalization of the Fibonacci and Lucas numbers, and using essentially a matrix approach they showed properties of these matrix sequences. The main motivation of this article is to study the matrix sequences of third-order Jacobsthal sequence and third-order Jacobsthal sequence.


\section {The third-order Jacobsthal, third-order Jacobsthal-Lucas sequences and their matrix sequences}

Now, considering these sequences, we define the matrix sequences which have elements of third-order Jacobsthal and third-order Jacobsthal-Lucas sequences.

\begin{definition}
Let $n\geq 0$. The third-order Jacobsthal matrix sequence $\{JM_{n}^{(3)}\}_{n\in \mathbb{N}}$ and third-order Jacobsthal-Lucas matrix sequence $\{jM_{n}^{(3)}\}_{n\in \mathbb{N}}$ are defined respectively by
\begin{equation}\label{d1}
JM_{n+3}^{(3)}=JM_{n+2}^{(3)}+JM_{n+1}^{(3)}+2JM_{n}^{(3)},
\end{equation}
\begin{equation}\label{d2}
jM_{n+3}^{(3)}=jM_{n+2}^{(3)}+jM_{n+1}^{(3)}+2jM_{n}^{(3)},
\end{equation}
with initial conditions $$JM_{0}^{(3)}=\left[
\begin{array}{ccc}
1& 0& 0 \\ 
0&1& 0\\ 
0 & 0&1%
\end{array}%
\right],\  JM_{1}^{(3)}=\left[
\begin{array}{ccc}
1& 1& 2 \\ 
1&0& 0\\ 
0 & 1&0%
\end{array}%
\right],\ JM_{2}^{(3)}=\left[
\begin{array}{ccc}
1& 3& 2 \\ 
1&1& 2\\ 
1 & 0&0%
\end{array}%
\right]$$
and $$jM_{0}^{(3)}=\left[
\begin{array}{ccc}
1& 4& 4 \\ 
2&-1& 2\\ 
1 & 1&-2%
\end{array}%
\right],\  jM_{1}^{(3)}=\left[
\begin{array}{ccc}
5& 5& 2 \\ 
1&4& 4\\ 
2 & -1&2%
\end{array}%
\right],\ jM_{2}^{(3)}=\left[
\begin{array}{ccc}
10& 7& 10 \\ 
5&5& 2\\ 
1 & 4&4%
\end{array}%
\right].$$
\end{definition}

In the rest of this paper, the third-order Jacobsthal and third-order Jacobsthal-Lucas matrix sequences will be denoted by $M_{J,n}^{(3)}$ and $M_{j,n}^{(3)}$ instead of $JM_{n}^{(3)}$ and $jM_{n}^{(3)}$, respectively.

\begin{theorem}\label{t1}
For $n\geq 0$, we have
\begin{equation}\label{p1}
\begin{aligned}
M_{J,n}^{(3)}&=\left(\frac{M_{J,2}^{(3)}+M_{J,1}^{(3)}+M_{J,0}^{(3)}}{(2-\omega_{1})(2-\omega_{2})}\right)2^{n}-\left(\frac{M_{J,2}^{(3)}-(2+\omega_{2})M_{J,1}^{(3)}+2\omega_{2}M_{J,0}^{(3)}}{(2-\omega_{1})(\omega_{1}-\omega_{2})}\right)\omega_{1}^{n}\\
&\ \ + \left(\frac{M_{J,2}^{(3)}-(2+\omega_{1})M_{J,1}^{(3)}+2\omega_{1}M_{J,0}^{(3)}}{(2-\omega_{2})(\omega_{1}-\omega_{2})}\right)\omega_{2}^{n}.
\end{aligned}
\end{equation}
\begin{equation}\label{p2}
\begin{aligned}
M_{j,n}^{(3)}&=\left(\frac{M_{j,2}^{(3)}+M_{j,1}^{(3)}+M_{j,0}^{(3)}}{(2-\omega_{1})(2-\omega_{2})}\right)2^{n}-\left(\frac{M_{j,2}^{(3)}-(2+\omega_{2})M_{j,1}^{(3)}+2\omega_{2}M_{j,0}^{(3)}}{(2-\omega_{1})(\omega_{1}-\omega_{2})}\right)\omega_{1}^{n}\\
&\ \ + \left(\frac{M_{j,2}^{(3)}-(2+\omega_{1})M_{j,1}^{(3)}+2\omega_{1}M_{j,0}^{(3)}}{(2-\omega_{2})(\omega_{1}-\omega_{2})}\right)\omega_{2}^{n}.
\end{aligned}
\end{equation}
\end{theorem}
\begin{proof}
(\ref{p1}): The solution of Eq. (\ref{d1}) is
\begin{equation}\label{e13}
M_{J,n}^{(3)}=c_{1}2^{n}+c_{2}\omega_{1}^{n}+c_{3}\omega_{2}^{n}.
\end{equation}
Then, let $M_{J,0}^{(3)}=c_{1}+c_{2}+c_{3}$, $M_{J,1}^{(3)}=2c_{1}+c_{2}\omega_{1}+c_{3}\omega_{2}$ and $M_{J,2}^{(3)}=4c_{1}+c_{2}\omega_{1}^{2}+c_{3}\omega_{2}^{2}$. Therefore, we have $(2-\omega_{1})(2-\omega_{2})c_{1}=M_{J,2}^{(3)}-(\omega_{1}+\omega_{2})M_{J,1}^{(3)}+\omega_{1}\omega_{2}M_{J,0}^{(3)}$, $(2-\omega_{1})(\omega_{1}-\omega_{2})c_{2}=M_{J,2}^{(3)}-(2+\omega_{2})M_{J,1}^{(3)}+2\omega_{2}M_{J,0}^{(3)}$, $(2-\omega_{2})(\omega_{1}-\omega_{2})c_{3}=M_{J,2}^{(3)}-(2+\omega_{1})M_{J,1}^{(3)}+2\omega_{1}M_{J,0}^{(3)}$. Using $c_{1}$, $c_{2}$ and $c_{3}$ in Eq. (\ref{e13}), we obtain
\begin{align*}
M_{J,n}^{(3)}&=\left(\frac{M_{J,2}^{(3)}+M_{J,1}^{(3)}+M_{J,0}^{(3)}}{(2-\omega_{1})(2-\omega_{2})}\right)2^{n}-\left(\frac{M_{J,2}^{(3)}-(2+\omega_{2})M_{J,1}^{(3)}+2\omega_{2}M_{J,0}^{(3)}}{(2-\omega_{1})(\omega_{1}-\omega_{2})}\right)\omega_{1}^{n}\\
&\ \ + \left(\frac{M_{J,2}^{(3)}-(2+\omega_{1})M_{J,1}^{(3)}+2\omega_{1}M_{J,0}^{(3)}}{(2-\omega_{2})(\omega_{1}-\omega_{2})}\right)\omega_{2}^{n}.
\end{align*}
(\ref{p2}): The proof is similar to the proof of (\ref{p1}).
\end{proof}

The following theorem gives us the $n$-th general term of the sequence given in (\ref{d1}) and (\ref{d2}).

\begin{theorem}\label{t2}
For $n\geq3$, we have
\begin{equation}\label{p3}
M_{J,n}^{(3)}=\left[
\begin{array}{ccc}
J_{n+1}^{(3)}& J_{n}^{(3)}+2J_{n-1}^{(3)}& 2J_{n}^{(3)} \\ 
J_{n}^{(3)}&J_{n-1}^{(3)}+2J_{n-2}^{(3)}& 2J_{n-1}^{(3)}\\ 
J_{n-1}^{(3)} & J_{n-2}^{(3)}+2J_{n-3}^{(3)}&2J_{n-2}^{(3)}%
\end{array}%
\right]
\end{equation}
\begin{equation}\label{p4}
M_{j,n}^{(3)}=\left[
\begin{array}{ccc}
j_{n+1}^{(3)}& j_{n}^{(3)}+2j_{n-1}^{(3)}& 2j_{n}^{(3)} \\ 
j_{n}^{(3)}&j_{n-1}^{(3)}+2j_{n-2}^{(3)}& 2j_{n-1}^{(3)}\\ 
j_{n-1}^{(3)} & j_{n-2}^{(3)}+2j_{n-3}^{(3)}&2j_{n-2}^{(3)}%
\end{array}%
\right]
\end{equation}
\end{theorem}
\begin{proof}
(\ref{p3}): Let use the principle of mathematical induction on $n$.
Let us consider $n=0$ in (\ref{ec:5}). We have $J_{-1}^{(3)}=0$, $J_{-2}^{(3)}=\frac{1}{2}$ and $J_{-3}^{(3)}=-\frac{1}{4}$. Then we write
$$M_{J,0}^{(3)}=\left[
\begin{array}{ccc}
J_{1}^{(3)}& J_{0}^{(3)}+2J_{-1}^{(3)}& 2J_{0}^{(3)} \\ 
J_{0}^{(3)}&J_{-1}^{(3)}+2J_{-2}^{(3)}& 2J_{-1}^{(3)}\\ 
J_{-1}^{(3)} & J_{-2}^{(3)}+2J_{-3}^{(3)}&2J_{-2}^{(3)}%
\end{array}%
\right]=\left[
\begin{array}{ccc}
1& 0& 0 \\ 
0&1& 0\\ 
0& 0&1%
\end{array}%
\right].$$
By iterating this procedure and considering induction steps, let us assume that the equality in (\ref{p3}) holds for all $n=k\in \mathbb{N}$. To finish the proof, we have to show that (\ref{p3}) also holds for $n=k+1$ by considering (\ref{ec:5}) and (\ref{d1}). Therefore we get
\begin{align*}
M_{J,k+2}^{(3)}&=M_{J,k+1}^{(3)}+M_{J,k}^{(3)}+2M_{J,k-1}^{(3)}\\
&=\left[
\begin{array}{ccc}
J_{k+2}^{(3)}+J_{k+1}^{(3)}+2J_{k}^{(3)}& J_{k+2}^{(3)}+2J_{k+1}^{(3)}& 2J_{k+1}^{(3)}+2J_{k}^{(3)}+4J_{k-1}^{(3)} \\ 
J_{k+1}^{(3)}+J_{k}^{(3)}+2J_{k-1}^{(3)}&J_{k+1}^{(3)}+2J_{k}^{(3)}& 2J_{k}^{(3)}+2J_{k-1}^{(3)}+4J_{k-2}^{(3)}\\ 
J_{k}^{(3)}+J_{k-1}^{(3)}+2J_{k-2}^{(3)} & J_{k}^{(3)}+2J_{k-1}^{(3)}&2J_{k-1}^{(3)}+2J_{k-2}^{(3)}+4J_{k-3}^{(3)}%
\end{array}%
\right]\\
&=\left[
\begin{array}{ccc}
J_{k+3}^{(3)}& J_{k+2}^{(3)}+2J_{k+1}^{(3)}& 2J_{k+2}^{(3)} \\ 
J_{k+2}^{(3)}&J_{k+1}^{(3)}+2J_{k}^{(3)}& 2J_{k+1}^{(3)}\\ 
J_{k+1}^{(3)} & J_{k}^{(3)}+2J_{k-1}^{(3)}&2J_{k}^{(3)}%
\end{array}%
\right].
\end{align*}
Hence we obtain the result.
If a similar argument is applied to (\ref{p4}), the proof is clearly seen.
\end{proof}

\begin{theorem}\label{t3}
Assume that $x\neq 0$. We obtain,
\begin{equation}\label{p5}
\sum_{k=0}^{n}\frac{M_{J,k}^{(3)}}{x^{k}}=\frac{1}{x^{n}\nu(x)}\left\lbrace
\begin{array}{c}
2M_{J,n}^{(3)}+\left(M_{J,n+2}^{(3)}-M_{J,n+1}^{(3)}\right)x+M_{J,n+1}^{(3)}x^{2}\\
-x^{n+1}\left(M_{J,2}^{(3)}-M_{J,1}^{(3)}-M_{J,0}^{(3)}-\left(M_{J,0}^{(3)}-M_{J,1}^{(3)}\right)x+M_{J,0}^{(3)}x^{2}\right)
\end{array}%
\right\rbrace,
\end{equation}
\begin{equation}\label{p6}
\sum_{k=0}^{n}\frac{M_{j,k}^{(3)}}{x^{k}}=\frac{1}{x^{n}\nu(x)}\left\lbrace
\begin{array}{c}
2M_{j,n}^{(3)}+\left(M_{j,n+2}^{(3)}-M_{j,n+1}^{(3)}\right)x+M_{j,n+1}^{(3)}x^{2}\\
-x^{n+1}\left(M_{j,2}^{(3)}-M_{j,1}^{(3)}-M_{j,0}^{(3)}-\left(M_{j,0}^{(3)}-M_{j,1}^{(3)}\right)x+M_{j,0}^{(3)}x^{2}\right)
\end{array}%
\right\rbrace,
\end{equation}
where $\nu(x)=x^{3}-x^{2}-x-2$.
\end{theorem}
\begin{proof}
In contrast, here we will just prove (\ref{p6}) since the proof of (\ref{p5}) can be done in a similar way. From Theorem \ref{t1}, we have
\begin{align*}
\sum_{k=0}^{n}\frac{M_{j,k}^{(3)}}{x^{k}}&=\left(\frac{M_{j,2}^{(3)}+M_{j,1}^{(3)}+M_{j,0}^{(3)}}{(2-\omega_{1})(2-\omega_{2})}\right)\sum_{k=0}^{n}\left(\frac{2}{x}\right)^{k}\\
&\ \ - \left(\frac{M_{j,2}^{(3)}-(2+\omega_{2})M_{j,1}^{(3)}+2\omega_{2}M_{j,0}^{(3)}}{(2-\omega_{1})(\omega_{1}-\omega_{2})}\right)\sum_{k=0}^{n}\left(\frac{\omega_{1}}{x}\right)^{k}\\
&\  \ + \left(\frac{M_{j,2}^{(3)}-(2+\omega_{1})M_{j,1}^{(3)}+2\omega_{1}M_{j,0}^{(3)}}{(2-\omega_{2})(\omega_{1}-\omega_{2})}\right)\sum_{k=0}^{n}\left(\frac{\omega_{2}}{x}\right)^{k}.
\end{align*}
By considering the definition of a geometric sequence, we get
\begin{align*}
\sum_{k=0}^{n}\frac{M_{j,k}^{(3)}}{x^{k}}&=\left(\frac{M_{j,2}^{(3)}+M_{j,1}^{(3)}+M_{j,0}^{(3)}}{(2-\omega_{1})(2-\omega_{2})}\right)\frac{2^{n+1}-x^{n+1}}{x^{n}(2-x)}\\
&\ \ - \left(\frac{M_{j,2}^{(3)}-(2+\omega_{2})M_{j,1}^{(3)}+2\omega_{2}M_{j,0}^{(3)}}{(2-\omega_{1})(\omega_{1}-\omega_{2})}\right)\frac{\omega_{1}^{n+1}-x^{n+1}}{x^{n}(\omega_{1}-x)}\\
&\  \ + \left(\frac{M_{j,2}^{(3)}-(2+\omega_{1})M_{j,1}^{(3)}+2\omega_{1}M_{j,0}^{(3)}}{(2-\omega_{2})(\omega_{1}-\omega_{2})}\right)\frac{\omega_{2}^{n+1}-x^{n+1}}{x^{n}(\omega_{2}-x)}\\
&=\frac{1}{x^{n}\nu(x)}\left\lbrace
\begin{array}{c}
\left(\frac{M_{j,2}^{(3)}+M_{j,1}^{(3)}+M_{j,0}^{(3)}}{(2-\omega_{1})(2-\omega_{2})}\right)(2^{n+1}-x^{n+1})(\omega_{1}-x)(\omega_{2}-x)\\
- \left(\frac{M_{j,2}^{(3)}-(2+\omega_{2})M_{j,1}^{(3)}+2\omega_{2}M_{j,0}^{(3)}}{(2-\omega_{1})(\omega_{1}-\omega_{2})}\right)(\omega_{1}^{n+1}-x^{n+1})(2-x)(\omega_{2}-x)\\
+ \left(\frac{M_{j,2}^{(3)}-(2+\omega_{1})M_{j,1}^{(3)}+2\omega_{1}M_{j,0}^{(3)}}{(2-\omega_{2})(\omega_{1}-\omega_{2})}\right)(\omega_{2}^{n+1}-x^{n+1})(2-x)(\omega_{1}-x)
\end{array}%
\right\rbrace,
\end{align*}
where $\nu(x)=x^{3}-x^{2}-x-2$. If we rearrange the last equality, then we obtain
$$
\sum_{k=0}^{n}\frac{M_{j,k}^{(3)}}{x^{k}}=\frac{1}{x^{n}\nu(x)}\left\lbrace
\begin{array}{c}
2M_{j,n}^{(3)}+\left(M_{j,n+2}^{(3)}-M_{j,n+1}^{(3)}\right)x+M_{j,n+1}^{(3)}x^{2}\\
-x^{n+1}\left(M_{j,2}^{(3)}-M_{j,1}^{(3)}-M_{j,0}^{(3)}-\left(M_{j,0}^{(3)}-M_{j,1}^{(3)}\right)x+M_{j,0}^{(3)}x^{2}\right)
\end{array}%
\right\rbrace.
$$ So, the proof is completed.
\end{proof}

In the following theorem, we give the sum of third-order Jacobsthal and third-order Jacobsthal-Lucas matrix sequences corresponding to different indices.
\begin{theorem}\label{t4}
For $r\geq m$, we have
\begin{equation}\label{p7}
\sum_{k=0}^{n}M_{J,mk+r}^{(3)}=\frac{1}{\sigma_{n}}\left\lbrace
\begin{array}{c}
M_{J,m(n+1)+r}^{(3)}-M_{J,r}^{(3)}+2^{m}M_{J,mn+r}^{(3)}-2^{m}M_{J,r-m}^{(3)}\\
-M_{J,m(n+1)+r}^{(3)}\mu(m)+M_{J,r}^{(3)}\mu(m)+M_{J,m(n+2)+r}^{(3)}-M_{J,r+m}^{(3)}
\end{array}%
\right\rbrace
\end{equation}
\begin{equation}\label{p8}
\sum_{k=0}^{n}M_{j,mk+r}^{(3)}=\frac{1}{\sigma_{n}}\left\lbrace
\begin{array}{c}
M_{j,m(n+1)+r}^{(3)}-M_{j,r}^{(3)}+2^{m}M_{j,mn+r}^{(3)}-2^{m}M_{j,r-m}^{(3)}\\
-M_{j,m(n+1)+r}^{(3)}\mu(m)+M_{j,r}^{(3)}\mu(m)+M_{j,m(n+2)+r}^{(3)}-M_{j,r+m}^{(3)}
\end{array}%
\right\rbrace,
\end{equation}
where $\sigma_{n}=2^{m+1}+(1-2^{m})(\omega_{1}^{m}+\omega_{2}^{m})-2$ and $\mu(m)=2^{m}+\omega_{1}^{m}+\omega_{2}^{m}$.
\end{theorem}
\begin{proof}
(\ref{p7}): Let us take $A=\frac{M_{J,2}^{(3)}+M_{J,1}^{(3)}+M_{J,0}^{(3)}}{(2-\omega_{1})(2-\omega_{2})}$, $B=\frac{M_{J,2}^{(3)}-(2+\omega_{2})M_{J,1}^{(3)}+2\omega_{2}M_{J,0}^{(3)}}{(2-\omega_{1})(\omega_{1}-\omega_{2})}$ and $C=\frac{M_{J,2}^{(3)}-(2+\omega_{1})M_{J,1}^{(3)}+2\omega_{1}M_{J,0}^{(3)}}{(2-\omega_{2})(\omega_{1}-\omega_{2})}$. Then, we write
\begin{align*}
\sum_{k=0}^{n}M_{j,mk+r}^{(3)}&=\sum_{k=0}^{n}(A2^{mk+r}-B\omega_{1}^{mk+r}+C\omega_{2}^{mk+r})\\
&=A2^{r}\sum_{k=0}^{n}2^{mk}-B\omega_{1}^{r}\sum_{k=0}^{n}\omega_{1}^{mk}+C\omega_{2}^{r}\sum_{k=0}^{n}\omega_{2}^{mk}\\
&=A2^{r}\left(\frac{2^{m(n+1)}-1}{2^{m}-1}\right)-B\omega_{1}^{r}\left(\frac{\omega_{1}^{m(n+1)}-1}{\omega_{1}^{m}-1}\right)+C\omega_{2}^{r}\left(\frac{\omega_{2}^{m(n+1)}-1}{\omega_{2}^{m}-1}\right)\\
&=\frac{1}{\sigma_{n}}\left\lbrace
\begin{array}{c}
\left(A2^{m(n+1)+r}-A2^{r}\right)\left(\omega_{1}^{m}\omega_{2}^{m}-(\omega_{1}^{m}+\omega_{2}^{m})+1\right)\\
-\left(B\omega_{1}^{m(n+1)+r}-B\omega_{1}^{r}\right)\left(2^{m}\omega_{2}^{m}-(2^{m}+\omega_{2}^{m})+1\right)\\
+\left(C\omega_{2}^{m(n+1)+r}-C\omega_{2}^{r}\right)\left(2^{m}\omega_{1}^{m}-(2^{m}+\omega_{1}^{m})+1\right)
\end{array}%
\right\rbrace,
\end{align*}
where $\sigma_{n}=2^{m+1}+(1-2^{m})(\omega_{1}^{m}+\omega_{2}^{m})-2$. After some algebra, we obtain
$$\sum_{k=0}^{n}M_{j,mk+r}^{(3)}=\frac{1}{\sigma_{n}}\left\lbrace
\begin{array}{c}
M_{J,m(n+1)+r}^{(3)}-M_{J,r}^{(3)}+2^{m}M_{J,mn+r}^{(3)}-2^{m}M_{J,r-m}^{(3)}\\
-M_{J,m(n+1)+r}^{(3)}\mu(m)+M_{J,r}^{(3)}\mu(m)+M_{J,m(n+2)+r}^{(3)}-M_{J,r+m}^{(3)}
\end{array}%
\right\rbrace,$$
where $\mu(m)=2^{m}+\omega_{1}^{m}+\omega_{2}^{m}$.

(\ref{p8}): The proof is similar to the proof of (\ref{p7}).
\end{proof}

\section{The relationships between matrix sequences $M_{J,n}^{(3)}$ and $M_{j,n}^{(3)}$}

\begin{lemma}\label{lem1}
For $m,n\in \mathbb{N}$, the third-order Jacobsthal and third-order Jacobsthal-Lucas matrix sequences are conmutative. The following results hold.
\begin{equation}\label{p10}
M_{J,n}^{(3)}M_{J,m}^{(3)}=M_{J,m}^{(3)}M_{J,n}^{(3)}=M_{J,n+m}^{(3)},
\end{equation}
\begin{equation}\label{p11}
M_{j,n}^{(3)}M_{j,m}^{(3)}=M_{j,m}^{(3)}M_{j,n}^{(3)},
\end{equation}
\begin{equation}\label{p12}
M_{j,1}^{(3)}M_{J,n}^{(3)}=M_{J,n}^{(3)}M_{j,1}^{(3)}=M_{j,n+1}^{(3)},
\end{equation}
\begin{equation}\label{p13}
M_{j,n}^{(3)}M_{J,1}^{(3)}=M_{J,1}^{(3)}M_{j,n}^{(3)}=M_{j,n+1}^{(3)},
\end{equation}
\begin{equation}\label{p14}
M_{J,n}^{(3)}M_{j,n+1}^{(3)}=M_{j,2n+1}^{(3)}.
\end{equation}
\end{lemma}
\begin{proof}
Here, we will just prove (\ref{p10}) and (\ref{p12}) since (\ref{p11}), (\ref{p13}) and (\ref{p14}) can be dealt with in the same manner. To prove Eq. (\ref{p10}), let us use the induction on $m$. If $m=0$, the proof is obvious since that $M_{J,0}^{(3)}$ is the identity matrix of order 3. Let us assume that Eq: (\ref{p10}) holds for all values $k$ less than or equal $m$. Now we have to show that the result is true for $m+1$:
\begin{align*}
M_{J,n+(m+1)}^{(3)}&=M_{J,n+m}^{(3)}+M_{J,n+m-1}^{(3)}+2M_{J,n+m-2}^{(3)}\\
&=M_{J,n}^{(3)}M_{J,m}^{(3)}+M_{J,n}^{(3)}M_{J,m-1}^{(3)}+2M_{J,n}^{(3)}M_{J,m-2}^{(3)}\\
&=M_{J,n}^{(3)}\left(M_{J,m}^{(3)}+M_{J,m-1}^{(3)}+2M_{J,m-2}^{(3)}\right)\\
&=M_{J,n}^{(3)}M_{J,m+1}^{(3)}.
\end{align*}
It is easy to see that $M_{J,n}^{(3)}M_{J,m}^{(3)}=M_{J,m}^{(3)}M_{J,n}^{(3)}$. Hence we obtain the result.

(\ref{p12}): To prove equation (\ref{p12}), we again use induction on $n$. Let $n=0$, we get $M_{j,1}^{(3)}M_{J,0}^{(3)}=M_{j,1}^{(3)}$. Let us assume that $M_{j,1}^{(3)}M_{J,n}^{(3)}=M_{j,n+1}^{(3)}$ is true for all values $k$ less than or equal $n$. Then,
\begin{align*}
M_{j,n+1}^{(3)}&=\left[
\begin{array}{ccc}
j_{n+2}^{(3)}& j_{n+1}^{(3)}+2j_{n}^{(3)}& 2j_{n+1}^{(3)} \\ 
j_{n+1}^{(3)}&j_{n}^{(3)}+2j_{n-1}^{(3)}& 2j_{n}^{(3)}\\ 
j_{n}^{(3)} & j_{n-1}^{(3)}+2j_{n-2}^{(3)}&2j_{n-1}^{(3)}%
\end{array}%
\right]\\
&=\left[
\begin{array}{ccc}
j_{n+1}^{(3)}& j_{n}^{(3)}+2j_{n-1}^{(3)}& 2j_{n}^{(3)} \\ 
j_{n}^{(3)}&j_{n-1}^{(3)}+2j_{n-2}^{(3)}& 2j_{n-1}^{(3)}\\ 
j_{n-1}^{(3)} & j_{n-2}^{(3)}+2j_{n-3}^{(3)}&2j_{n-2}^{(3)}%
\end{array}%
\right]\left[
\begin{array}{ccc}
1& 1& 2 \\ 
1&0& 0\\ 
0& 1&0%
\end{array}%
\right]\\
&=M_{j,n}^{(3)}M_{J,1}^{(3)}\\
&=M_{j,1}^{(3)}M_{J,n-1}^{(3)}M_{J,1}^{(3)}\\
&=M_{j,1}^{(3)}M_{J,n}^{(3)}.
\end{align*}
Hence the result. 
\end{proof}

\begin{theorem}\label{t5}
For $m,n\in \mathbb{N}$ the following properties hold.
\begin{equation}\label{p15}
M_{j,n}^{(3)}=M_{J,n}^{(3)}+4M_{J,n-1}^{(3)}+4M_{J,n-2}^{(3)},
\end{equation}
\begin{equation}\label{p16}
M_{j,n}^{(3)}=2M_{J,n+1}^{(3)}-M_{J,n}^{(3)}+2M_{J,n-1}^{(3)},
\end{equation}
\begin{equation}\label{p17}
M_{j,1}^{(3)}M_{J,n}^{(3)}=M_{J,n+2}^{(3)}+3M_{J,n}^{(3)}+2M_{J,n-1}^{(3)},
\end{equation}
\end{theorem}
\begin{proof}
First, here, we will just prove (\ref{p15}) and (\ref{p17}) since (\ref{p16})ncan be dealt with in the same manner. So, if we consider the right-hand side of equation (\ref{p15}) and use Theorem \ref{t2}, we get
\begin{align*}
M_{J,n}^{(3)}&+4M_{J,n-1}^{(3)}+4M_{J,n-2}^{(3)}\\
&=\left[
\begin{array}{ccc}
J_{n+1}^{(3)}& J_{n}^{(3)}+2J_{n-1}^{(3)}& 2J_{n}^{(3)} \\ 
J_{n}^{(3)}&J_{n-1}^{(3)}+2J_{n-2}^{(3)}& 2J_{n-1}^{(3)}\\ 
J_{n-1}^{(3)} & J_{n-2}^{(3)}+2J_{n-3}^{(3)}&2J_{n-2}^{(3)}%
\end{array}%
\right]+4\left[
\begin{array}{ccc}
J_{n}^{(3)}& J_{n-1}^{(3)}+2J_{n-2}^{(3)}& 2J_{n-1}^{(3)} \\ 
J_{n-1}^{(3)}&J_{n-2}^{(3)}+2J_{n-3}^{(3)}& 2J_{n-2}^{(3)}\\ 
J_{n-2}^{(3)} & J_{n-3}^{(3)}+2J_{n-4}^{(3)}&2J_{n-3}^{(3)}%
\end{array}%
\right]\\
&\ \ +4\left[
\begin{array}{ccc}
J_{n-1}^{(3)}& J_{n-2}^{(3)}+2J_{n-3}^{(3)}& 2J_{n-2}^{(3)} \\ 
J_{n-2}^{(3)}&J_{n-3}^{(3)}+2J_{n-4}^{(3)}& 2J_{n-3}^{(3)}\\ 
J_{n-3}^{(3)} & J_{n-4}^{(3)}+2J_{n-5}^{(3)}&2J_{n-4}^{(3)}%
\end{array}%
\right]\\
&=\left[
\begin{array}{ccc}
j_{n+1}^{(3)}& j_{n}^{(3)}+2j_{n-1}^{(3)}& 2j_{n}^{(3)} \\ 
j_{n}^{(3)}&j_{n-1}^{(3)}+2j_{n-2}^{(3)}& 2j_{n-1}^{(3)}\\ 
j_{n-1}^{(3)} & j_{n-2}^{(3)}+2j_{n-3}^{(3)}&2j_{n-2}^{(3)}%
\end{array}%
\right]\\
&=M_{j,n}^{(3)}.
\end{align*}
From Eq. (\ref{e5}), $j_{n}^{(3)}=J_{n}^{(3)}+4J_{n-1}^{(3)}+4J_{n-1}^{(3)}$, as required in (\ref{p15}).

Second, let us consider the left-hand side of Eq. (\ref{p17}). Using Theorem \ref{t2}, we write $$M_{j,1}^{(3)}M_{J,n}^{(3)}=\left[
\begin{array}{ccc}
j_{2}^{(3)}& j_{1}^{(3)}+2j_{0}^{(3)}& 2j_{1}^{(3)} \\ 
j_{1}^{(3)}&j_{0}^{(3)}+2j_{-1}^{(3)}& 2j_{0}^{(3)}\\ 
j_{0}^{(3)} & j_{-1}^{(3)}+2j_{-2}^{(3)}&2j_{-1}^{(3)}%
\end{array}%
\right]\left[
\begin{array}{ccc}
J_{n+1}^{(3)}& J_{n}^{(3)}+2J_{n-1}^{(3)}& 2J_{n}^{(3)} \\ 
J_{n}^{(3)}&J_{n-1}^{(3)}+2J_{n-2}^{(3)}& 2J_{n-1}^{(3)}\\ 
J_{n-1}^{(3)} & J_{n-2}^{(3)}+2J_{n-3}^{(3)}&2J_{n-2}^{(3)}%
\end{array}%
\right].$$ From matrix production, we have
\begin{align*}
M_{j,1}^{(3)}M_{J,n}^{(3)}&=\left[
\begin{array}{ccc}
5& 5& 2\\ 
1&4& 4\\ 
2 &-1&2%
\end{array}%
\right]\left[
\begin{array}{ccc}
J_{n+1}^{(3)}& J_{n}^{(3)}+2J_{n-1}^{(3)}& 2J_{n}^{(3)} \\ 
J_{n}^{(3)}&J_{n-1}^{(3)}+2J_{n-2}^{(3)}& 2J_{n-1}^{(3)}\\ 
J_{n-1}^{(3)} & J_{n-2}^{(3)}+2J_{n-3}^{(3)}&2J_{n-2}^{(3)}%
\end{array}%
\right]\\
&=\left[
\begin{array}{ccc}
J_{n+3}^{(3)}& J_{n+2}^{(3)}+2J_{n+1}^{(3)}& 2J_{n+2}^{(3)} \\ 
J_{n+2}^{(3)}&J_{n+1}^{(3)}+2J_{n}^{(3)}& 2J_{n+1}^{(3)}\\ 
J_{n+1}^{(3)} & J_{n}^{(3)}+2J_{n-1}^{(3)}&2J_{n}^{(3)}%
\end{array}%
\right]+3\left[
\begin{array}{ccc}
J_{n+1}^{(3)}& J_{n}^{(3)}+2J_{n-1}^{(3)}& 2J_{n}^{(3)} \\ 
J_{n}^{(3)}&J_{n-1}^{(3)}+2J_{n-2}^{(3)}& 2J_{n-1}^{(3)}\\ 
J_{n-1}^{(3)} & J_{n-2}^{(3)}+2J_{n-3}^{(3)}&2J_{n-2}^{(3)}%
\end{array}%
\right]\\
&\ \ +2\left[
\begin{array}{ccc}
J_{n}^{(3)}& J_{n-1}^{(3)}+2J_{n-2}^{(3)}& 2J_{n-1}^{(3)} \\ 
J_{n-1}^{(3)}&J_{n-2}^{(3)}+2J_{n-3}^{(3)}& 2J_{n-2}^{(3)}\\ 
J_{n-2}^{(3)} & J_{n-3}^{(3)}+2J_{n-4}^{(3)}&2J_{n-3}^{(3)}%
\end{array}%
\right]\\
&=M_{J,n+2}^{(3)}+3M_{J,n}^{(3)}+2M_{J,n-1}^{(3)}.
\end{align*}
Hence the result.
\end{proof}

\begin{theorem}\label{t6}
For $m,n\in \mathbb{N}$, the following properties hold.
\begin{equation}\label{p20}
M_{J,m}^{(3)}M_{j,n+1}^{(3)}=M_{j,n+1}^{(3)}M_{J,m}^{(3)}=M_{j,m+n+1}^{(3)},
\end{equation}
\begin{equation}\label{p21}
\left(M_{j,n+1}^{(3)}\right)^{m}=\left(M_{j,1}^{(3)}\right)^{m}M_{J,mn}^{(3)}.
\end{equation}
\end{theorem}
\begin{proof}
(\ref{p20}): Let us consider the left-hand side of equation (\ref{p20}) and Lemma \ref{lem1} and Theorem \ref{t5}. We have
\begin{align*}
M_{J,m}^{(3)}M_{j,n+1}^{(3)}&=M_{J,m}^{(3)}M_{j,1}^{(3)}M_{J,n}^{(3)}\\
&=M_{J,m}^{(3)}\left(2M_{J,2}^{(3)}-M_{J,1}^{(3)}+2M_{J,0}^{(3)},\right)M_{J,n}^{(3)}\\
&=2M_{J,m+n+2}^{(3)}-M_{J,m+n+1}^{(3)}+2M_{J,m+n}^{(3)}\\
&=\left(2M_{J,2}^{(3)}-M_{J,1}^{(3)}+2M_{J,0}^{(3)},\right)M_{J,m+n}^{(3)}.
\end{align*}
Moreover, from Eq. (\ref{p16}) in Theorem \ref{t5}, we obtain $$M_{J,m}^{(3)}M_{j,n+1}^{(3)}=M_{j,1}^{(3)}M_{J,m}^{(3)}M_{J,n}^{(3)}=M_{j,m+1}^{(3)}M_{J,m}^{(3)}.$$ Also, from Lemma \ref{lem1}, it is seen that $M_{J,m}^{(3)}M_{j,n+1}^{(3)}=M_{j,m+n+1}^{(3)}$ which finishes the proof of (\ref{p20}).

(\ref{p21}): To prove equation (\ref{p21}), let us follow induction steps on $m$. For $m=1$, the proof is clear by Lemma \ref{lem1}. Now, assume that it is true for all positive integers $m$, that is, $\left(M_{j,n+1}^{(3)}\right)^{m}=\left(M_{j,1}^{(3)}\right)^{m}M_{J,mn}^{(3)}$. 

Therefore, we have to show that it is true for $m+1$. If we multiply this $m$-th step by $M_{j,n+1}^{(3)}$ on both sides from the right, then we have
\begin{align*}
\left(M_{j,n+1}^{(3)}\right)^{m+1}&=\left(M_{j,1}^{(3)}\right)^{m}M_{J,mn}^{(3)}M_{j,n+1}^{(3)}\\
&=\left(M_{j,1}^{(3)}\right)^{m}M_{J,mn}^{(3)}M_{j,1}^{(3)}M_{J,n}^{(3)}\\
&=\left(M_{j,1}^{(3)}\right)^{m}M_{j,1}^{(3)}M_{J,mn}^{(3)}M_{J,n}^{(3)}\\
&=\left(M_{j,1}^{(3)}\right)^{m+1}M_{J,mn+n}^{(3)}\\
&=\left(M_{j,1}^{(3)}\right)^{m+1}M_{J,(m+1)n}^{(3)}
\end{align*}
which finishes the induction and gives the proof of (\ref{p21}).
\end{proof}

\begin{corollary}
For $n\geq 0$, by taking $m=2$ and $m=3$ in the Eq. (\ref{p21}) given in Theorem \ref{t6}, we obtain
\begin{equation}\label{p22}
\left(M_{j,n+1}^{(3)}\right)^{2}=\left(M_{j,1}^{(3)}\right)^{2}M_{J,2n}^{(3)}=M_{j,1}^{(3)}M_{j,2n+1}^{(3)},
\end{equation}
\begin{equation}\label{p23}
\left(M_{j,n+1}^{(3)}\right)^{3}=\left(M_{j,1}^{(3)}\right)^{3}M_{J,3n}^{(3)}=\left(M_{j,1}^{(3)}\right)^{2}M_{j,3n+1}^{(3)}.
\end{equation}
\end{corollary}

\begin{corollary}
For $n \in \mathbb{N}$, we have the following result
\begin{equation}\label{p24}
\begin{aligned}
\left(j_{n+1}^{(3)}\right)^{2}+\left(j_{n}^{(3)}\right)^{2}+4j_{n}^{(3)}j_{n-1}^{(3)}&=34J_{2n+1}^{(3)}+43J_{2n}^{(3)}+34J_{2n-1}^{(3)}\\
&=5j_{2n+2}^{(3)}+5j_{2n+1}^{(3)}+2j_{2n}^{(3)}.
\end{aligned}
\end{equation}
\end{corollary}
\begin{proof}
The proof can be easily seen by the coefficient in the first row and column of the matrix $\left(M_{j,n+1}^{(3)}\right)^{2}$ in (\ref{p22}) and the Eq. (\ref{d2}).
\end{proof}
\section{Conclusions}
In this paper, we study a generalization of the Jacobsthal and Jacobsthal-Lucas matrix sequences. Particularly, we define the third-order Jacobsthal and third-order Jacobsthal-Lucas matrix sequences, and we find some combinatorial identities. As seen in \cite{Cook-Bac} one way to generalize the Jacobsthal recursion is as follows $$J_{n+r}^{(r)}=\sum_{k=0}^{r-1}J_{n+r-k}^{(r)}+2J_{n}^{(r)},$$
with $n\geq 0$ and initial conditions $J_{k}^{(r)}$, for $k=0,1,...,r-2$ and $J_{r-1}^{(r)}=1$, has characteristic equation $(x-2)(x^{r-1}+x^{r-2}+\cdots +1)=0$ with eigenvalues 2 and $\omega_{k}=e^{\frac{2\pi i m}{r}}$, for $k=0,1,...,r-1$. It would be interesting to introduce the higher order Jacobsthal and Jacobsthal-Lucas matrix sequences. Further investigations for these and other methods useful in discovering identities for the higher order Jacobsthal and Jacobsthal-Lucas sequences will be addressed in a future paper.


\end{document}